\documentclass[sigconf]{acmart}

\usepackage[utf8]{inputenc}

\usepackage{amsmath,amssymb}
\usepackage{amsthm}
\usepackage{mathrsfs}
\usepackage{stmaryrd}
\usepackage{enumerate}
\usepackage[algoruled,vlined,english,linesnumbered]{algorithm2e}
\usepackage{comment}
\usepackage{multirow}
\usepackage{xspace}

\usepackage{mathtools}
\mathtoolsset{showonlyrefs,showmanualtags}

\usepackage[inline]{enumitem}

\definecolor{input}{HTML}{303060}
\definecolor{output}{HTML}{804000}
\definecolor{string}{HTML}{A02020}
\definecolor{parent}{HTML}{A020A0}
\definecolor{function}{HTML}{205080} 
\definecolor{constructor}{HTML}{205080}
\definecolor{method}{HTML}{205080}
\definecolor{keyword}{HTML}{008000}
\definecolor{error}{HTML}{B01010}
\definecolor{comment}{HTML}{60A060}

\newcommand{\cIn}{{\color{input} \tt \phantom{C}In}:}
\newcommand{\cOut}{{\color{output} \tt Out}:}

\newcommand{\noopsort}[1]{}

\DeclareMathOperator{\val}{val}

\DeclareMathOperator{\lcm}{lcm}

\newcommand{\N}{\mathbb N}
\newcommand{\Z}{\mathbb Z}
\newcommand{\NN}{\mathbb N}
\newcommand{\ZZ}{\mathbb Z}

\newcommand{\Q}{\mathbb Q}
\newcommand{\QQ}{\mathbb Q}
\newcommand{\Qp}{\Q_p}

\newcommand{\RR}{\mathbb R}
\newcommand{\CC}{\mathbb C}

\newcommand{\FF}{\mathbb{F}}

\newcommand{\Skel}{\mathrm{Skel}}
\newcommand{\mon}{\textsc{mon}}

\newcommand{\X}{\mathbf{X}}
\newcommand{\Y}{\mathbf{Y}}
\renewcommand{\i}{\mathbf{i}}
\renewcommand{\j}{\mathbf{j}}
\renewcommand{\r}{\mathbf{r}}

\newcommand{\ifnonempty}[3]{%
  \def\tempa{}%
  \def\tempb{#1}%
  \ifx\tempa\tempb 
  #3 
  \else            
  #2
  \fi}

\newcommand{\Kz}{K^\circ}
\newcommand{\KzX}[1][]{K\{ \X \ifnonempty{#1}{; #1}{} \}^\circ}
\newcommand{\KX}[1][]{K \{ \X \ifnonempty{#1}{; #1}{} \}}
\newcommand{\Lz}{L^\circ}
\newcommand{\LzX}[1][]{L\{ \X \ifnonempty{#1}{; #1}{} \}^\circ}
\newcommand{\LX}[1][]{L \{ \X \ifnonempty{#1}{; #1}{} \}}
\newcommand{\TzX}[1][]{T\{ \X \ifnonempty{#1}{; #1}{} \}^\circ}
\newcommand{\TX}[1][]{T \{ \X \ifnonempty{#1}{; #1}{} \}}
\newcommand{\TT}{\mathbb T}
\newcommand{\TTzX}[1][]{\TT\{ \X \ifnonempty{#1}{; #1}{} \}^\circ}
\newcommand{\TTX}[1][]{\TT \{ \X \ifnonempty{#1}{; #1}{} \}}
\newcommand{\LKX}[1][]{\eta^{\NN} \KX[#1]}

\newcommand{\Kb}{\bar{K}}

\newcommand{\LT}{LT}

\newcommand{\sage}{\textsc{SageMath}\xspace}

\definecolor{purple}{rgb}{0.6,0,0.6}

\definecolor{answer}{rgb}{0,0.5,0.2}

\newtheorem{theo}{Theorem}[section]
\newtheorem{lem}[theo]{Lemma}
\newtheorem{prop}[theo]{Proposition}

\theoremstyle{definition}
\newtheorem{rem}[theo]{Remark}
\newtheorem{ex}[theo]{Example}
\newtheorem{deftn}[theo]{Definition}

\begin{document}

\title{Gröbner bases over Tate algebras}

\author{Xavier Caruso}
  \affiliation{Université de Bordeaux,
  \institution{CNRS, INRIA}
  \city{Bordeaux, France}}
  \email{xavier.caruso@normalesup.org}

\author{
Tristan Vaccon}
  \affiliation{Universit\'e de Limoges;
  \institution{CNRS, XLIM UMR 7252}
  \city{Limoges, France}  
  \postcode{87060}  
  }
  \email{tristan.vaccon@unilim.fr}

\author{
Thibaut Verron}
  \affiliation{
  \institution{Johannes Kepler University\\Institute for Algebra }
  \city{Linz, Austria}  
}
\email{thibaut.verron@jku.at}

\thanks{The third author is supported by the Austrian FWF grant F5004.}

\begin{abstract}
Tate algebras, introduced in~\cite{Tate}, are fundamental objects in the 
context of analytic geometry over the $p$-adics. Roughly speaking, they 
play the same role as polynomial algebras play in classical algebraic 
geometry. In the present article, we develop the formalism of Gröbner 
bases for Tate algebras. We prove an analogue of the Buchberger criterion 
in our framework and design a Buchberger-like and a F4-like algorithm for 
computing Gröbner bases over Tate algebras. 
An implementation in \sage is also discussed.
\end{abstract}

\ccsdesc[500]{Computing methodologies~Algebraic algorithms}


\vspace{-1.5mm}
\terms{Algorithms, Theory}

\keywords{Algorithms, Power series, Tate algebra, Gröbner bases, F4 
algorithm, $p$-adic precision}

\maketitle

\section{Introduction}

In complex geometry, the concept of analytic functions is obviously a 
notion of first importance. They form a class of functions that exhibit 
strong rigidity properties as polynomials do but, at the same time, allow 
for many analytic constructions such as taking limits, integrals, \emph{etc.} 
For this reason, they often appear as a bridge between algebra and 
analysis.

For many arithmetical applications, the completion $\Qp$ of $\QQ$ is often 
as relevant as $\RR$ or $\CC$. At the beginning of the 20th century, 
mathematicians realized that it would be quite interesting to 
develop the theory of $p$-adic analytic functions and eventually that of 
$p$-adic analytic geometry. However doing so is not an easy task owing to 
the unpleasant topology on $\Qp$, which is totally disconnected.

In~\cite{Tate}, Tate proposed to replace the classical $p$-adic topology 
by some well-suited Grothendieck topology and came up with the notion of
$p$-adic rigid variety.
Basically, the construction of rigid varieties follows that of schemes 
in algebraic geometry. They are obtained by gluing pieces ---~the 
so-called \emph{affinoids}~--- with respect to the aforementioned 
Grothendieck topology. As for affinoids, they are defined as the 
``spectrum'' of quotients of some particular algebras, called \emph{Tate 
algebras}. Thereby, Tate algebras play the same role in rigid geometry as 
polynomial algebras do in classical algebraic geometry.

From the purely algebraic point of view, Tate algebras have been widely 
studied and it has been demonstrated that they share some properties with 
polynomial algebras~\cite{BGR84}. 
However, as far as we know, the computational aspects 
of Tate algebras have not been developed yet.
This contrasts with the polynomial setting, for which we have at our 
disposal the theory of Gröbner bases~\cite{Bu65,Cox15}, which has become 
over the years a research topic on its own. 
The aim of the present article is to extend the notion of Gröbner bases
to Tate algebras.

Some difficulties need to be overcome. The most significant one is that 
elements in Tate algebras are, by nature, infinite convergent series and 
so they do not have a degree. This seems to be a serious obstruction since 
the degree is the most basic notion on which the classical theory of 
Gröbner bases is built. However, analyzing the definition of Tate 
algebras, we notice that a Tate series defines a sequence of 
\emph{polynomials} (of growing degrees) by reduction modulo $p^n$ when $n$ 
varies. In order to take advantage of this observation, we introduce an 
order on the terms taking into account the $p$-adic valuation of the 
coefficients. This order is not well-founded as classical term orders are 
usually. However, we shall prove that it is topologically well-founded (in 
the sense that every decreasing sequence tends to $0$) and that this 
weaker property is enough to guarantee the termination of our algorithms
in the finite precision model.

\vspace{1em}

\noindent
{\it Related works.}
Gröbner bases over rings ---~and in particular over $\ZZ$ 
and $\ZZ/n\ZZ$~--- have also received some attention~\cite{AL,KC}.
These developments are of course related to this article since quotients 
of Tate algebras are polynomial algebras over $\ZZ/p^n \ZZ$ for $n$ 
varying.
The main difference between our point of view and that of 
\emph{loc.~cit.} appears in the choice of the term ordering; while, in 
the theory of Gröbner bases of rings, only the degree is considered,
our setting forces us to include the valuation of the coefficients
in the definition of the term ordering. It is the ``price to pay''
to be able to pass smoothly to the completion and catch inexact 
bases as $\ZZ_p$ or~$\QQ_p$.

The special term ordering we use comes from two different 
sources. The first one 
is the theory of tropical Gröbner bases by Chan and Maclagan~\cite{CM} in 
which, for the first time, the valuation of the coefficients has been 
taken into account in the definition of the term ordering.
Later on, Vaccon and his coauthors~\cite{Vaccon:these,Vaccon:2015,Vaccon:2017,
Vaccon:2018} observed that tropical
orders are relevant for the computation of $p$-adic Gröbner bases as
they improve substantially the numerical accuracy. The definition of
our term order is the natural outcome of this observation.
Our second source of inspiration is the theory of standard bases, which 
was designed originally to ``compute'' the singularties of algebraic 
varieties~\cite{Mora,Grabe}. 
This theory introduces the notion of term order of
local/mixed type, on which the term ordering we are using in the 
present article is modeled.

\vspace{1em}

\noindent
{\it Structure of the article.}
In \S\ref{sec:Tate}, we introduce Tate algebras and develop the
theory of Gröbner bases over them. We prove in particular the
existence of finite Gröbner bases and study their structure.
\S\ref{sec:algo} is devoted to algorithms. We first design a 
variant of the Buchberger algorithm that runs over Tate algebras.
Several results towards its numerical stability are also presented.
We then move to F4-like algorithms and show how they could be
adapted to fit into the framework of Tate algebras. 
Finally, in \S\ref{sec:impl}, an implementation in \sage is
briefly discussed.

\vspace{1em}

\noindent
{\it Notations.}
The notation $\mathbb{N}$ will refer to the
set of nonnegative integers (including~$0$).
If $\mathfrak A$ is a ring, we will denote its group of invertible
elements by $\mathfrak A^\times$.
We fix a positive integer $n$.
Let $X_1, \ldots, X_n$ be $n$ variables. We
will use the short notation $\X$ for $(X_1,\dots,X_n)$. Similarly
for $\i = (i_1, \ldots, i_n) \in \NN^n$, we shall write $\X^\i$ for 
$X_1^{i_1} \cdots X_n^{i_n}$.

\section{Gröbner bases over Tate algebras}
\label{sec:Tate}

Throughout this article, we fix a field $K$ equipped
with a discrete valuation $\val : K \to \Z \sqcup \{+\infty\}$,
normalized by $\val(K^\times) = \Z$. We shall always assume that
$K$ is complete with respect to the distance defined
by $\val$. We let $\Kz$ be the subring of $K$ consisting of
elements of nonnegative valuation and $\pi$ be a uniformizer 
of $K$, that is an element of valuation $1$. We set
$\Kb = \Kz/\pi\Kz$.

A typical example of $K$ as above is the field of $p$-adic numbers 
$\QQ_p$ (equipped with the $p$-adic valuation). For this example, we 
have $\Kz = \ZZ_p$ and $\Kb = \FF_p$.

\subsection{Tate algebras}

We endow $\RR^n$ with the usual scalar product.

\begin{deftn}
\label{def:KXr}
  Let $\r = (r_{1},\dots,r_{n}) \in \QQ^{n}$.
  The \emph{Tate algebra} $\KX[\r]$ is defined by:
  \begin{equation}
    \label{eq:1}
    \KX[\r] := \left\{ \sum_{\mathbf{i} \in \NN^{n}} a_{\mathbf{i}}\X^{\mathbf{i}} 
    \text{ s.t. }
       a_{\mathbf{i}}\in K \text{ and } 
       \val(a_\i) - \r{\cdot}\i \xrightarrow[|\i| \rightarrow +\infty]{} +\infty
    \right\}
  \end{equation}
  The tuple $\mathbf{r}$ is called the convergence log-radii of the Tate algebra.
\end{deftn}

\noindent
Elements of $\KX[\r]$ are the power series converging on the product
of \emph{closed} balls $B(0,|\pi|^{r_{1}}) \times \dots \times 
B(0,|\pi|^{r_{n}})$ where $|\cdot|$ is the absolute value on $K$
induced by $\val$.
When $\r = (0, \ldots, 0)$, we will simply write $\KX$ instead of
$\KX[(0,\ldots,0)]$. 

\begin{ex}
  \label{ex:1}
  Let $K=\QQ_{p}$.
  The series $f_{1} = \frac{1}{p} + X + p X^{2} + p^{2} X^{3} + \dots$ lies in $K\{X\}$.
  The series $f_{2} = 1 + X + X^{2} + X^{3} + \dots$ does not lie in $K\{X\}$, because it does not converge when evaluated at $1$ (for example).
  However, it does converge when evaluated at $x$ with $|x|<1$, so it lies in $K\{X;(r)\}$ for all negative~$r$.
\end{ex}

The Tate algebra $\KX[\r]$ is equipped with the Gauss valuation
$\val_\r : \KX[\r] \to \Q \sqcup \{+\infty\}$ defined as follows:
$$\val_\r\Big(\sum_{\i \in\N^n} a_\i X^\i\Big) =
\min_{\i\in \N^n} \val(a_\i) - \r{\cdot}\i.$$
We observe that the minimum is always reached thanks to the growth
condition imposed in Definition~\ref{def:KXr}. Moreover, the image of 
$\val_\r$ is discrete. Geometrically, the Gauss valuation corresponds
to the minimal valuation reached by the series on its domain of
convergence (possibly after a finite extension of $K$).

\begin{deftn}
The \emph{integral Tate algebra ring} $\KzX[\r]$ is defined as the
subring of $\KX[\r]$ consisting of elements with nonnegative Gauss
valuation.
\end{deftn}

\noindent
Again we will use the notation $\KzX$ for $\KzX[(0, \ldots, 0)]$.
When $\mathbf{r} \in \ZZ^{n}$, observe that
$\KX[\r] = K\left\{ \pi^{r_{1}}X_1, \dots, \pi^{r_{n}}X_n \right\}$
and similarly for $\KzX[\r]$. The case $r \in \ZZ^n$ then reduces to
$\r = 0$ \emph{via} a change of variables.

\begin{ex}
  With the notations of Example~\ref{ex:1}, $f_{1}$ does not lie in $K\{X\}^{\circ}$, but $f_{2}$ does lie in $K\{X;r\}^{\circ}$.
\end{ex}

\begin{prop}
  We have $\KX[\r] = \KzX[\r] \left[ \frac{1}{\pi} \right]$.
\end{prop}

\subsection{About terms}

From now on, we fix a log-radii $\r \in \QQ^n$.

\subsubsection*{Monoids of terms.}

We first recall some basic definitions.

\begin{deftn}
A \emph{monoid} is a set equipped with a single associative binary 
operation, which has a neutral element.

An \emph{ideal} of a monoid $M$ is a subset $I \subset M$ such
that, for all $a \in M$ and $x \in I$, we have $ax \in I$.
\end{deftn}

We define the monoid of terms $\TX[\r]$ as the multiplicative monoid 
consisting of the elements $a \X^\mathbf{i}$ with $a \in K^\times$ and
$\i \in \NN^n$.
We let also $\TzX[\r]$ be the submonoid of $\TX[\r]$ consisting of terms $a 
\X^\mathbf{i}$ for which $\val_\r(a \X^\mathbf{i}) \geq 0$.
The multiplicative group $K^\times$ (resp. $(\Kz)^\times$) embeds 
into $\TX[\r]$ (resp. $\TzX[\r]$). We set:
$$\TTX[\r] = \TX[\r] / K^\times
\quad \text{and} \quad
\TTzX[\r] = \TzX[\r] / (\Kz)^\times.$$
The inclusion $\TzX[\r] \subset \TX[\r]$ induces a canonical 
morphism (which is no longer injective) $\TTzX[\r] \to \TTX[\r]$.
The ideals of $\TTX[\r]$ (resp. of $\TTzX[\r]$) are in bijective
correspondance with the ideals of $\TX[\r]$ (resp. of $\TzX[\r]$).
Moreover, $\TTX[\r]$ and $\TTzX[\r]$ do not contain non trivial invertible 
elements. In other words, the divisibility relation defines an order on 
$\TTX[\r]$ and $\TTzX[\r]$. The following lemma elucidates the structure 
of $\TTX[\r]$ and $\TTzX[\r]$.

\begin{lem}
\label{lem:monterms}
(1) The mapping $\TTX[\r] \to \N^n$, $a \X^\i \mapsto \i$ is an
isomorphism of monoids.

\noindent
(2) The mapping $\TTzX[\r] \to \Q^+ \times \N^n$, 
$a \X^\i \mapsto (\val_\r(a\X^\i), \i)$ is an injective morphism
of monoids; its image is included in $\frac 1D \N \times \N^n$
where $D$ is a common denominator of the coordinates of $\r$.

\noindent
(3) The natural morphism $\TTzX[\r] \to \TTX[\r]$ corresponds
to the projection onto the factor $\N^n$.
\end{lem}

\begin{prop}
\label{prop:skel}
Let $I$ be an ideal of $\TTX[\r]$ (resp. of $\TTzX[\r]$).
Then there exists a unique subset $S$ of $I$ having the two following
properties: 
(1)~$S$ generates $I$, and
(2)~every subset generating $I$ contains~$S$.
Moreover $S$ is finite.
\end{prop}

\begin{proof}
The unicity is easy. Indeed if $S$ and $S'$ satisfy~(1) and~(2),
one must have $S \subset S'$ and $S' \subset S$, \emph{i.e.} $S = S'$.
In order to prove the existence, we define $S$ as the set of minimal 
elements of $I$ for the divisibility relation. 
The fact that $S$ generates $I$ 
follows from the fact that divisibility 
is a well-funded order on $\TTX[\r]$ (\emph{cf} Lemma~\ref{lem:monterms}).
The point~(2) is obvious.

It remains to prove that $S$ is finite. For this, we observe that
any sequence with values in $\N$ necessarily has a nondecreasing
subsequence. Extracting subsequences repeatedly, we find that the
previous property also holds for sequences with values in $\N^m$ for
any integer $m$. By Lemma~\ref{lem:monterms}, it also holds for
sequences with values in $\TTX[\r]$ (resp. in $\TTzX[\r]$). Therefore,
if $S$ were not finite, we would be able to extract from $S$ a
nondecreasing sequence. This contradicts the fact that $S$ is
composed by minimal elements.
\end{proof}

\begin{deftn}
\label{def:skel}
Let $I$ be an ideal of $\TTX[\r]$ (resp. of $\TTzX[\r]$).
The subset $S$ of Proposition~\ref{prop:skel} is called the
\emph{skeleton} of $I$; it is denoted by $\Skel(I)$.

The \emph{skeleton} of an ideal of $\TX[\r]$ (resp. of $\TzX[\r]$)
is defined as the skeleton of its image in $\TTX[\r]$
(resp. in $\TTzX[\r]$); it is denoted by $\Skel(I)$.
\end{deftn}

In what follows, it will sometimes be convenient to work more generally 
with fractional ideals. By definition a \emph{fractional ideal} of 
$\TTzX[\r]$ is a subset of $\TTX[\r]$ which is stable by multiplication 
by elements in $\TTzX[\r]$. 
The notion of skeleton can be extended to fractional ideals $I$ of 
$\TTzX[\r]$ for which there exists $N \in \NN$ such that $I \subset 
\pi^{-N} \TTzX[\r]$. For such ideals, $\Skel(I)$ is a finite subset of
$\TX[\r]/(\Kz)^\times$.
An interesting example of fractional ideal is:
\begin{equation}
\label{eq:TTxv}
\TTX[\r]^{\geq v} = \big\{\, t \in \TTX[\r] \text{ s.t. }
\val_\r(t) \geq v\, \big\}.
\end{equation}

\vspace{-2ex}

\begin{rem}
The effective computation of $\Skel(\TTX[\r]^{\geq v})$ is not
an easy problem. It has been solved for $n=1$ in \cite{CL} using
the theory of continued fractions. It would be interesting to 
generalize the results of \emph{loc.~cit.} to higher~$n$.
\end{rem}

\subsubsection*{Term order.}

We fix a \emph{monomial order} $\leq_\omega$ on $\N^n$. We recall that 
this means that $\leq_\omega$ is a well-order which is compatible with the 
addition. Usual examples of monomial orders are $\text{lex}$, 
$\text{grevlex}$, \emph{etc.}

\begin{deftn}
We define a preorder $\leq$ on $\TX[\r]$, $\TzX[\r]$ by:
$$\begin{array}{rr@{\hspace{2ex}}l}
a \X^\mathbf{i} \leq b \X^\j 
 & \text{iff} & \val_\r (a \X^\i) > \val_\r (b \X^\j) \\
 & \text{or}  & \val_\r (a \X^\i) = \val_\r (b \X^\j)
\text{ and } \i \leq_\omega \j.
\end{array}$$
\end{deftn}

\begin{rem}
The inequality sign is reversed in the first line:
we require that $\val_\r (a \X^\i) > \val_\r (b \X^\j)$ and not
$\val_\r (a \X^\i) < \val_\r (b \X^\j)$. This is not a typo and
will be important in the sequel.
\end{rem}

We underline that $\leq$ is not antisymmetric (and so not an order). 
More precisely, for $t_1, t_2 \in \TX[\r]$, the fact that $t_1 \leq t_2$ 
and $t_2 \leq t_1$ is equivalent to the existence of $a \in (\Kz)^\times$
such that $t_1 = a t_2$.
As a consequence, $\leq$ induces an order on $\TTzX[\r]$.
On the contrary, we draw the attention of the reader that $\leq$ does not 
factor through $\TTX[\r]$.

\begin{ex}
  \label{ex:2}
  Let $K = \QQ_{p}$ and consider $K\{X,Y\}$ with the lexicographical order.
  The preorder $\leq$ orders terms as follows:
  \begin{multline}
    \label{eq:4}
    \cdots > XY^{2} > XY > X > \cdots > Y > 1 > \cdots \\
    \cdots > pXY^{2} > \cdots > p > \cdots > p^{2}XY^{2} > \cdots
  \end{multline}
  The terms $\X^\mathbf{i}$ and $-\X^\mathbf{i}$ are ``equal'' for 
  $\leq$.
  So are $\X^\mathbf{i}$ and $(1{+}p)\X^\mathbf{i}$.
\end{ex}

It is easily seen that the preorder $\leq$ is total. 
In turns out that it is not a well-order since the infinite sequence
$(p^n)_{n \geq 0}$ is strictly decreasing. Nevertheless, we have:

\begin{lem}
\label{lem:strictly_dec_seq}
Let $(t_j)_{j \in \N}$ be a strictly decreasing
sequence in $\TX[\r]$ (resp. in $\TzX[\r]$).
Then $\lim_{j \to \infty} \val_\r(t_j) = +\infty$.
\end{lem}

\begin{proof}
From the definition of $\leq$, it follows that the sequence
$(\val_\r(t_j))_{j \in \N}$ is nondecreasing. Moreover it takes
its values in $\frac 1 D \N$ for some positive integer $D$. 
Finally, the fact that $\leq_\omega$ is a well-order implies that
for each fixed $v \in \frac 1 D \N$, there is only a finite number 
of indices $j$ for which $\val_\r(t_j) = v$. Combining these inputs,
we find that $\val_\r(t_j)$ must tend to $+\infty$.
\end{proof}

We notice that if $\i \neq \j$, the terms $a_\i \X^\i$ and $a_\j
X^\j$ are never ``equal'' for $\leq$. Therefore, any nonzero series
$f = \sum_{\i\in \N^n} a_\i X^\i \in \KX[\r]$ has a unique leading
term. We denote it $\LT(f)$. 

\begin{ex}
  With the notations of Example~\ref{ex:2}, 
  the leading term of 
  $g_{2} = XY + p + p^2 XY$ is $\LT(g_{2}) = (1{+}p^2)XY$.
\end{ex}

\subsection{Gröbner bases}

\begin{deftn}
  Given an ideal $J$ of $\KX[\r]$ (resp. of $\KzX[\r]$), 
  we denote by $\LT(J)$ the subset of $\TX[\r]$ (resp. of $\TzX[\r]$)
  consisting of elements of the form $\LT(f)$ with $f \in J$, 
  $f \neq 0$.
\end{deftn}

We check immediately that $\LT(J)$ is an ideal of the monoid 
$\TX[\r]$ (resp. of $\TzX[\r]$).

\begin{deftn}
Let $J$ be an ideal of $\KX[\r]$ (resp. of $\KzX[\r]$).
A family $(g_1,\dots,g_s) \in J^s$ is a Gröbner basis (in short, GB)
of $J$ if $LT(J)$ is generated by the $LT(g_i)$'s in $\TX[\r]$ (resp.
$\TzX[\r]$).
\end{deftn}

\begin{prop}
\label{prop:GBgen}
Let $G = (g_1,\dots,g_s)$
be a GB of an ideal $J$ of $\KX[\r]$ (resp. of $\KzX[\r]$).
Then $G$ generates $J$. 
\end{prop}

\begin{proof}
Let $f \in J$. We define inductively a sequence $(f_j)_{j \in \mathbb{N}}$ 
as follows.
Let $f_0=f$. Given~$j$, we write $LT(f_j)=a_j \X^{\mathbf{i}_j} 
LT(g_{i_j})$ and define $f_{j+1}=f_j-a_j 
\X^{\mathbf{i}_j} g_{i_j}$. Then $LT(f_{j+1}) < LT (f_j)$.
By Lemma~\ref{lem:strictly_dec_seq}, $\val_\r(LT(f_j)) = \val_\r(f_j)$ 
goes to infinity when $j$ goes to infinity.
Therefore we can then write $f= \sum_j a_j \X ^{\mathbf{i}_j} g_{i_j}$
as a converging series. By regrouping terms,
we get $f \in \left\langle g_1,\dots, g_s \right\rangle$.
\end{proof}

Proposition~\ref{prop:skel} gives a lot of information about the 
ideal $\LT(J)$ (where $J$ is an ideal of $\KX[\r]$ or $\KzX[\r]$). 
These results have interesting consequences on Gröbner bases.

\begin{theo}
Any ideal of $\KX[\r]$ or $\KzX[\r]$ has a finite GB.
\end{theo}

\begin{proof}
Let $t_1, \ldots, t_s$ be the elements of $\Skel(\LT(J))$.
For all $i$, let $g_i \in J$ be such that $\LT(g_i) = t_i$ in
$\TTX[\r]$ (resp. in $\TTzX[\r]$). Then $(g_1, \ldots, g_s)$
is a GB of $J$.
\end{proof}

\begin{rem}
Combining the previous theorem with Proposition~\ref{prop:GBgen},
we obtain that any ideal of $\KX[\r]$ (resp. of $\KzX[\r]$) is
finitely generated. In other words, we have proved that the rings 
$\KX[\r]$ and $\KzX[\r]$ are Noetherian (which was of course
already known for a long time).
\end{rem}

\noindent
Another important consequence of Proposition~\ref{prop:skel}
is the notion of minimal GB that we discuss now.

\begin{deftn}
Let $J$ be an ideal of $\KX[\r]$ (resp. of $\KzX[\r]$).
A GB $G = (g_1, \ldots, g_s)$ is \emph{minimal} if
the images in $\TTX[\r]$ (resp. in $\TTzX[\r]$) of the $\LT(g_i)$'s
are exactly the elements of $\Skel(\LT(J))$, with no repetition.
\end{deftn}

A direct consequence of the definition is that two minimal GB of 
a given ideal $J$ have the same cardinality, namely the cardinality of 
$\Skel(\LT(J))$.
Proposition~\ref{prop:skel} also implies the next theorem.

\begin{theo}
\label{theo:minimalGB}
Let $J$ be an ideal of $\KX[\r]$ (resp. of $\KzX[\r]$).
Let $G$ be a GB of $J$. Then, there exists a subset $G' \subset
G$ which is a minimal GB
 of $J$.
\end{theo}

\subsection{Comparison results}

So far, we have defined a notion of GB for ideals of $\KX[\r]$ 
and $\KzX[\r]$. The aim of this subsection is to compare them.

\begin{prop}
Let $I$ be an ideal of $\KzX[\r]$ and let $G$ be a GB of 
$I$. Then $G$ is a GB of the ideal $J = I\big[\frac 1 
\pi\big]$ of $\KX[\r]$.
\end{prop}

\begin{rem}
Note that minimality of GB is not preserved when 
passing from $\KzX[\r]$ to $\KX[\r]$. For example, $G = (p,X)$ 
is a minimal GB of the ideal $I = (p,X)$ of $\Kz\{X\}$.
However it is not a minimal GB of $J = I\big[\frac 1
\pi\big] = K\{X\}$ since $p$ divides $X$ in this ring.
\end{rem}

Going in the other direction (\emph{i.e.} from $\KX[\r]$ to $\KzX[\r]$) 
is more subtle. First of all, we remark that, if we start with an ideal 
$J$ of $\KX[\r]$, there exist many ideals $I$ of $\KzX[\r]$ with the 
property that $I \big[\frac 1\pi\big] = J$. However, the set of such 
ideals~$I$ 
has a unique maximal element (for the inclusion); it is the ideal 
$J^\circ = J \cap \KzX[\r]$. This special ideal $J^\circ$ can also be 
caracterized by the fact that it is $\pi$-saturated.

\begin{prop}
\label{prop:KtoKo}
Let $J$ be an ideal of $\KX[\r]$ and let $G = (g_1, \ldots, g_s)$ be a 
GB (resp. a minimal GB) of $J$. We assume that $\val_\r(g_i) = 0$ for 
all $i$. Then $G$ is a GB (resp. a minimal GB) of $J^\circ$.
\end{prop}

\begin{proof}
Let $G$ be a GB of $J$.
Let $t \in \LT(J^\circ)$. Then $t$ is a multiple of one of the 
$\LT(g_i)$'s in $\TX[\r]$. Since $\val_\r(g_i) = 0$, we deduce
that $\LT(g_i)$ divides $t$ in $\TzX[\r]$ as well.
Consequentlt $G$ is a GB of $J^\circ$.
The fact that minimality is preserved is easy.
\end{proof}

When $\r \in \ZZ^n$, it is easy to build a GB of $J$ satisfying
the assumption of Proposition~\ref{prop:KtoKo} from any GB of $J$.
Indeed if $(g_1, \ldots, g_s)$ is a GB of $J$ then $\val_\r(g_i)$
is an integer for all $i$ and the family
$(\pi^{-\val_\r(g_1)} g_1, \ldots, \pi^{-\val_\r(g_s)} g_s)$ is
a GB of $J$.
On the contrary, when $\r \not\in \ZZ^n$, the problem is more
complicated as illustrated by the next example.

\begin{ex}
Choose $n = 1$ and $\r = (\frac 1 2)$ and let $J$ be ideal of
$K\{X\}$ generated by $X$. The ideal $J^\circ$ is then generated
by $g_1 = \pi X$ and $g_2 = \pi X^2$. More precisely, one checks
that $(g_1, g_2)$ is a minimal GB of $J^\circ$. In particular, we
observe that the cardinality of a minimal GB of $J$ does not agree
with that of a minimal GB of $J^\circ$.
\end{ex}

\noindent
For a general $r \in \QQ^n$, Proposition~\ref{prop:KtoKo} can be
refined as follows.

\begin{prop}
Let $J$ be an ideal of $\KzX[\r]$ and let $G = (g_1, \ldots, g_s)$ 
be a GB of $J$. 
Then a GB of $J^\circ$ is $(t_{i,j}{\cdot}g_i)$'s where, 
for each fixed~$i$, the $t_{i,j}$'s enumerate the elements of
$\Skel\big(\TX[\r]^{\geq -\val_\r(g_i)}\big)$ (\emph{cf} 
Eq.~\eqref{eq:TTxv}).
\end{prop}

\subsubsection*{Reduction in the residue field.}

When $\r = (0,\ldots,0)$, the quotient $\KzX/\pi\KzX$ is isomorphic
to the polynomial algebra $\Kb[\X]$, on which we have a well-defined
notion of Gröbner bases.

\begin{prop}
Let $J$ be an ideal of $\KX$. Set $J^\circ = J \cap \KzX$ and let $\bar 
J^\circ$ be the image of $J^\circ$ in $\Kb[\X]$. Let $g_1, \dots, g_s$ 
in $J$ be such that $\val_0(g_i)=1$ and let $\bar g_1, \ldots, \bar g_s$ 
be their images in $\bar J^\circ$.
Then the following assertions are equivalent:

\noindent
(1)~$(g_1,\dots,g_s)$ is a GB of $J$;

\noindent
(2)~$(g_1,\dots,g_s)$ is a GB of $J^\circ$;

\noindent
(3)~$(\bar g_1, \dots, \bar g_s)$ is a GB of $\bar J^\circ$.
\end{prop}

\begin{proof}
The equivalence between (1) and~(2) has been already proved.
We now prove that (2) implies (3).
Let $\bar{f} \in \bar J^\circ$ and let $f \in J^\circ$ be a lift
of $\bar f$.
We can write
$LT(f)=a X^\mathbf{i} LT(g_i)$ for some 
$a,\mathbf{i}$ and $i$.
Then $LT(\bar{f}) = \bar{a}X^\mathbf{i} LT(\bar g_i)$.
Therefore the $LT(\bar g_i)$'s generate $\LT(\bar J^\circ)$.
We prove finally that (3) implies (2).
Let $f \in J^\circ$.
Set $h = \pi^{-\val_0(f)} f$. Clearly $h \in J$ and $h \in \KzX$.
Thus $h \in J^\circ$. By (3), we can write
$LT(\bar h) = \bar{a} X^\mathbf{i} LT(\bar g_i)$ for
$\bar a \in \Kb$ and $\i \in \N^n$.
We write $LT(h)=h_0 X^H$
with $h_0 \in (\Kz)^\times$ and similarly, $LT(g_i)=b_0 X^F$
with $b_0 \in (\Kz)^\times$. Then $X^F$ divides $X^H$.
Let $L$ be such that $X^H=X^F \cdot X^L$.
Then \[LT(h)= h_0 b_0^{-1} X^L LT(g_i)\]
with $\frac{h_0}{b_0} \in \Kz$.
This concludes the proof.
\end{proof}

\section{Algorithms}
\label{sec:algo}

\subsection{Division and membership test}

Not surprisingly, Gröbner bases can be used to test membership in 
ideals. Before going further in this direction, we need to adapt the 
division algorithm to our setting.
We will need two variants depending on where we are looking for the 
quotients.

\begin{prop}
\label{prop:division}
Let $f, h_1, \dots, h_m \in \KX[\r]$.
Then, there exist $q_1,\dots,q_m \in \KX[\r]$ (resp. $q_1,\dots,q_m \in 
\KzX[\r]$) and $r \in \KX[\r]$ such that:

\noindent
(1)~$f= q_1 h_1 + \dots+q_m h_m+r$,

\noindent
(2)~for all $i$ and all terms $t$ of $r$, $LT(h_i) \nmid t$ in $\TX[\r]$ 
(resp. in $\TzX[\r]$),

\noindent
(3)~for all terms $t_i$ of $q_i$, we have $LT(t_i h_i) \leq LT(f)$.
\end{prop}

\begin{proof}
We only give the proof of $\KX[\r]$, the case of $\KzX[\r]$ being
totally similar. We will construct by induction sequences
$(f_j)_{j \geq 0}$, $(q_{i,j})_{j \geq 0}$ ($1 \leq i \leq m$) 
and $(r_j)_{j \geq 0}$ such that:
\begin{equation}
\label{eq:division}
f = f_j + q_{1,j}h_1 + \dots + q_{m,j}h_j + r_j.
\end{equation}
We set $f_0=f$, $r_0=0$ and $q_{1,0}=\dots=q_{m,0}=0$.
If $LT(f_j)$ is divisible by some $LT(h_{i_j})$, we set $f_{j+1} = 
f_j-\frac{LT(f_j)}{LT(h_i)}h_i$ and $q_{i_j,j+1} = q_{i_j,j} + 
\frac{LT(f_j)}{LT(h_i)}$, and leave unchanged $r$ and the others 
$q_i$'s.
Otherwise, we set $f_{j+1}=f_j-LT(f_j)$ and $r_{j+1}=r_j+LT(f_j)$.

If follows from the construction that $LT(f_{j+1}) < LT(f_j)$ 
for all $j$. By Lemma~\ref{lem:strictly_dec_seq}, 
$\lim_{j \to \infty} \val_\r(f_j) = +\infty$, 
\emph{i.e.} $(f_j)_{j \geq 0}$ converges to $0$ in $\KX[\r]$.
Besides, $\val_\r\big(\frac{LT(f_j)}{LT(h_i)}\big)$ tends to
infinity as well, so that the sequences $(q_{i,j})_{j \geq 0}$ 
all converge. Combining this with Eq.~\eqref{eq:division}, we find
that $(r_j)_{j \geq 0}$ also converges.
The elements $q_i = \lim_{j \to \infty} q_{i,j}$ and
$r = \lim_{j \to \infty} r_j$ satisfy the requirements of the
proposition.
\end{proof}

Algorithm~\ref{algo:reduction} below summarizes the proof of
Proposition~\ref{prop:division}.
\begin{algorithm} 

  \SetKwInOut{Input}{input}\SetKwInOut{Output}{output}

  \Input{$f, h_{1},\dots,h_m \in \KX[\r]$}
  \Output{$q_1, \ldots, q_m, r$ satisfying Prop.~\ref{prop:division}}

  $r, q_1, \ldots, q_m \leftarrow 0$\;
  \While{$f \neq 0$}{
    \While{$\exists i \in \{1,\dots,m\}$ such that $\LT(h_{i}) \mid \LT(f)$}{
      $q_i \leftarrow q_i + \frac{\LT(f)}{\LT(h_{i})}$\;
      $f \leftarrow f - \frac{\LT(f)}{\LT(h_{i})} h_{i}$\;
    }
    $r \leftarrow r + \LT(f)$\;
    $f \leftarrow f - \LT(f)$\;
  }
  \textbf{Return} $q_1, \ldots, q_m, r$\;

  \caption{\texttt{division} \label{algo:reduction}}
\end{algorithm}
In general, it does not terminate, keeping computing
more and more accurate approximations of the $q_i$'s and $r$.
However, in the common case where the coefficients of the input 
series are all known up to finite precision, \emph{i.e.} modulo 
$\pi^N$ for some $N$, Algorithm~\ref{algo:reduction} does terminate.

\begin{rem}
\label{rem:finiteprec}
When working at finite precision, it is more intelligent, instead of 
computing the quotient $\frac{\LT(f)}{\LT(h_i)}$ (which would possibly
lead to losses of precision), to choose an \emph{exact} term $t$ 
such that the equality $\LT(f) = t \cdot 
\LT(h_i)$ holds at the working precision, and 
use it on lines~4 and~5.
Doing so, we limit the losses of precision.
\end{rem}

In general, the conditions of Proposition~\ref{prop:division} are not 
enough to determine uniquely the $q_i$'s and $r$. However, 
Proposition~\ref{prop:divisionunicity} below provides a
weak unicity result when $(h_1, \ldots, h_m)$ is a Gröbner bases,
which can be used to test membership.

\begin{prop}
\label{prop:divisionunicity}
Let $J$ be an ideal of $\KX[\r]$ (resp. of $\KzX[\r]$) and
let $(g_1, \ldots, g_s)$ be a GB of $J$.
Let $f \in \KX[\r]$. We assume that we are given a decomposition
$f = q_1 g_1 + \cdots + q_s g_s + r$
satisfying the requirements of Proposition~\ref{prop:division}.
Then $r = 0$ if and only if $f \in J$.
\end{prop}

\begin{proof}
The ``only if'' is clear.
Conversely, assume by contradiction that 
$f \in J$ and $r \neq 0$. Then $\LT(r)$ makes 
sense. From the conditions of Proposition~\ref{prop:division}, we
deduce that $\LT(r)$ is not divisible by $\LT(g_i)$ for all $i$.
Hence $\LT(r) \not\in \LT(J)$. This is contradiction since
$r \in J$.
\end{proof}

\begin{rem}
In the integral Tate algebra setting, it is not true that the remainder
in the division by Gröbner bases is unique. For example, the
division in $\KzX$ of $f = 1+p$ by $h = p$ can be written either
$f = 0\times h + (1{+}p)$ or $f = 1\times h + 1$.
This is a general limitation of Gröbner bases over rings, even in the polynomial case~\cite{AL}.
\end{rem}

\subsection{Buchberger's algorithm}
\label{ssec:buchberger}

In this subsection, we adapt Buchberger's algorithm to fit into
the framework of Tate algebras.
The adaptation is more or less straightforward except on two points. 
The first one is related to finite precision, as already encountered 
previously. The second point is of different nature; it is related to 
the fact that, when the log-radii are not integers, the crucial notion 
of S-polynomials is not well-defined as the monoid $\TzX[\r]$ does 
not admit $\gcd$'s.
In what follows, we will give satisfying answers to these issues.

\subsubsection*{Buchberger's criterion.}

To begin with, we assume $\r = (0,\ldots,0)$.
Under this hypothesis, the monoid of terms $\TX$ admits $\gcd$'s 
and $\lcm$'s. Concretely we define:
\begin{align*}
\gcd(a \X^\i, \, b\X^\j) 
 & = \pi^{\min(\val(a), \val(b))} X^{\inf(\i,\j)}, \\
\lcm(a \X^\i, \, b\X^\j) 
 & = \pi^{\max(\val(a), \val(b))} X^{\sup(\i,\j)}
\end{align*}
where the $\inf$ and the $\sup$ over $\NN^n$ are taken coordinate
by coordinate. In what follows, in order to simplify notations, we will 
write $\val$ instead of $\val_{(0,\ldots,0)}$. If $t_1$ and $t_2$ are
two terms, the valuation of $\gcd(t_1,t_2)$ (resp. of $\lcm(t_1,t_2)$)
is the minimum (resp. the maximum) of $\val(t_1)$ and $\val(t_2)$.

\begin{deftn}
For $f, g$ in $\KX$,
we define: 
\[S(f,g) = \frac{LT(g)}{\gcd(LT(f),LT(g))}f - 
\frac{LT(f)}{\gcd(LT(f),LT(g))}g. \]
\end{deftn}

\medskip

We have the following classical lemma:

\begin{lem}
\label{lem:cancel_and_Spol}
Let $h_1,\dots,h_m \in \KX$ and $t_1,\dots,t_m \in \TX$.
We assume that the $\LT(t_i h_i)$'s all have the same image in
$\TX/(\Kz)^\times$ and that
$LT(\sum_{i=1}^m t_i h_i) < LT(t_i h_i)$. Then 
\[\sum_{i=1}^m t_i h_i =
\sum_{i=1}^{m-1} t'_i {\cdot}S(h_i,h_{i+1})
+ t'_m {\cdot}h_m  \]
for some $t'_1, \ldots, t'_m \in \TX$ such that
$\val(t'_m h_m) > \val(t_1 h_1)$ and
$\val (t'_i) + \max(\val(h_i), \val(h_{i+1})) \geq \val(t_1 h_1)$
for $i \in \{1, \ldots, m{-}1\}$.
\end{lem}

\begin{theo}
\label{theo:Buchberger_criterion}
Let $h_1, \ldots, h_s$ be elements of $\KX$ (resp. of $\KzX$) and let 
$I$ be the ideal of $\KX$ (resp. of $\KzX$) generated by the $h_i$'s.
Then $(h_1,\dots,h_s)$ is a GB of $I$ if and only if all $S(h_i,h_j)$, 
$i \neq j$, reduce to zero after division by $(h_1,\dots,h_s)$ using 
Algorithm~\ref{algo:reduction}.
\end{theo}

\begin{proof}
The ``only if'' part follows from Proposition~\ref{prop:divisionunicity}.
We prove the ``if'' part.
Let us assume by contradiction
that there exists some $f \in I$
such that $LT(f) \notin \left\langle LT(h_i) \right\rangle$.
We can write $f = \sum_i q_i h_i$ with $q_i \in \KX$ (resp. $q_i \in
\KzX$). Define $t = \max_i \LT(q_i h_i)$.
We have $LT(f) < t$ because of the hypothesis that $LT(f) \notin 
\left\langle LT(h_i) \right\rangle$. We can moreover assume that
the decomposition $f = \sum_i q_i h_i$ is chosen in such a way
that $t$ is minimal.

Let $J$ be the set of indices $i$ for which $\LT(q_i h_i) = a{\cdot}t$ for
some $a \in (\Kz)^\times$.
Set $t_i = LT(q_i)$ for $i \in J$ and define
$h =\sum_{i \in J} t_i h_i$; we have $\LT(h) < t$.
Applying Lemma \ref{lem:cancel_and_Spol}, we find $j_0 \in J$ and
terms $t'$, $t'_{j,k}$ (for $j,k \in J$) such that:
\[
h = \sum_{j,k \in J} t'_{j,k} S(h_j,h_k) + t' h_{j_0}\]
and $\val(t' h_{j_0}) > \val(h)$, $\val(t'_{j,k}) + \min(\val(h_j),
\val(h_j)) \geq \val(h)$.
Applying Proposition~\ref{prop:division} with the S-polynomials,
and using the fact that the leading terms
of the summands in an S-polynomial cancel out, 
we get $b_1, \ldots, b_m \in \KX$ such that 
$h = \sum_{i=1}^m b_i h_i$
and $LT(b_i h_i)<t$ for all $i$.
Therefore, we find that $f$ can be written as
$f=\sum_{i \in \mathbf{i}} q'_i h_i$ with $q'_1, \ldots, q'_m
\in \KX$ and $LT(q_i' h_i)<t$ for all~$i$. This contradicts the
minimality of~$t$.
\end{proof}

\subsubsection*{Buchberger's algorithm.}

After Theorem~\ref{theo:Buchberger_criterion}, it is easy to design a 
Buchberger type algorithm for computing GB over $\KX$ and $\KzX$. It is 
Algorithm~\ref{algo:buchberger}.
\begin{algorithm} 
  \SetKwInOut{Input}{input}\SetKwInOut{Output}{output}
  \Input{$f_1,\dots,f_m$ in $\KX$ (resp. in $\KzX$)}
  \Output{a GB $G$ of the ideal of $\KX$ (resp. of $\KzX$) 
  generated by the $f_i$'s}

  $G \leftarrow \{f_{1},\dots,f_{m}\}$;\,
  $B \leftarrow \{ (f_i, f_j), 1 \leq i < j \leq m \}$\;
  \While{$B \neq \emptyset$}{
    $(f,g) \leftarrow $ element of $B$; $B \leftarrow B \setminus \{(f,g)\}$\;
    $h \leftarrow$ $S$-polynomial of $f$ and $g$\;
    $\_, r \leftarrow \texttt{division}(h, G)$\;
    \If{$r \neq 0$}{
      $B \leftarrow B \cup \{(g,r) \text{ for } g \in G\}$;\,
      $G \leftarrow G \cup \{r\}$
    }
  }
  \textbf{Return $G$}
  
  \caption{Buchberger's algorithm} \label{algo:buchberger}
\end{algorithm}
Studying its termination is a bit subtle. Indeed, we have already
seen that Algorithm~\ref{algo:reduction} does not terminate in
general when we are working at infinite precision. Therefore, 
Algorithm~\ref{algo:buchberger} does not terminate either (since
it calls Algorithm~\ref{algo:reduction} on line~5).
Nevertheless, one may observe that if, instead of calling
Algorithm~\ref{algo:reduction}, we ask the reduced form of $h$ 
modulo $G$ to an oracle that answers instantly, then 
Algorithm~\ref{algo:buchberger} does terminate. 
In other terms, the only source of possible infinite loops in 
Algorithm~\ref{algo:buchberger} comes from 
Algorithm~\ref{algo:reduction}.

Of course, this point of view is purely theoretical and not
satisfying in practice. 
In practice, the coefficients of $f_1, \ldots, f_m$ are given at 
finite precision, \emph{i.e.} modulo $\pi^N$ for some integer $N$,
and all the computations are carried out at finite precision.
In this setting, we have seen that Algorithm~\ref{algo:reduction}
does terminate, so Algorithm~\ref{algo:buchberger} also terminates.
The counterpart is that it is \emph{a priori} not clear that the 
result output by Algorithm~\ref{algo:buchberger} is a correct 
approximation of a GB of the ideal we started with.
Nevertheless, in the case of $\KzX$, this property holds true as
precised by the following theorem.

\begin{theo}
\label{theo:prec}
Let $I$ be an ideal of $\KzX$
and let $(f_1, \ldots, f_m)$ be a generating family of $I$.
Let also $N$ be an integer such that $N > \val(t)$ for all 
$t \in \Skel(\LT(I))$.

When Algorithm~\ref{algo:buchberger} is called with 
$f_1 + O(\pi^N), \ldots, f_m + O(\pi^N)$, it outputs
$G = (g_1, \ldots, g_s)$ with the following properties:

\noindent
(1)~each $g_i$ is known at precision at least $O(\pi^N)$, and

\noindent
(2)~$G$ is the approximation of an actual GB of $I$.
\end{theo}

\begin{proof}
The fact that the precision on the $g_i$'s does not decrease 
follows from the fact that Algorithm~\ref{algo:buchberger} only
performs ``exact'' divisions (\emph{cf} Remark~\ref{rem:finiteprec}).

We now prove~(2). Since the $g_j$'s are obtained as linear combinations 
of the inputs $f_i + O(\pi^N)$, there exist $\hat g_1, \ldots, \hat g_s 
\in I$ such that $g_i = \hat g_i + O(\pi^N)$ for all~$i$.
We set $\hat G = (\hat g_1, \ldots, \hat g_s)$; it is enough to
prove that $\hat G$ is a GB of $I$.

Let $I_N = I + \pi^N \KzX$ and $\hat G_N = (\hat g_1, \ldots, \hat g_s, \pi^N)$.
We claim that $\hat G_N$ is a GB of $I_N$. Since it generates $I_N$, it 
is enough to check Buchberger's criterion. By construction, we know that 
the reduction of $S(\hat g_i, \hat g_j)$ modulo $\hat G$ is a multiple 
of $\pi^N$. Hence $S(\hat g_i, \hat g_j)$ reduces to zero modulo $\hat 
G_N$. On the other hand, it follows from the definition of
S-polynomials that $S(\hat 
g_i, \pi^N)$ is divisible by $\pi^N$; hence it also reduces to $0$ 
modulo $\hat G_N$. The claim is proved.

Let $t \in \LT(I_N)$. Then $t = \LT(f + \pi^N h)$ for some $f \in I$
and some $h \in \KzX$. If $\val(f) < N$, we have $t = \LT(f) \in
\LT(I)$. Otherwise $t$ is a multiple of $\pi^N$. We have then
proved that $\LT(I_N)$ is the ideal generated by $\LT(I)$ and
the term $\pi^N$. This implies that, if
$H$ is a GB of $I$, then $H_N = H \cup \{\pi^N\}$ is a GB of $I_N$.
Moreover by our assumption on $\Skel(\LT(I))$, if $H$ is minimal
then $H_N$ is also.

Choose now a \emph{minimal} GB $H$ of $I$.
From what we have done before and Theorem~\ref{theo:minimalGB}, 
it follows that $\LT(H_N) \subset \LT(\hat G_N)$. 
Besides, since the $g_i$'s do not vanish at 
precision $O(\pi^N)$, we have $\val(\hat g_i) < N$ for all~$i$.
Consequently, 
$\LT(H) \subset \LT(\hat G)$. 
In particular $\LT(\hat G)$ generates $\LT(I)$, and so
$\hat G$ is a GB of $I$.
\end{proof}

In the case of $\KX$, we cannot hope to have similar guarantees.
Indeed, if we ask from the GB of the ideal I generated by $f_1 = X + 
O(\pi^N)$ and $f_2 = X + O(\pi^N)$, the answer might be either $(X)$ if 
$f_1 = f_2 = X$, or $(1)$ if $f_1 = X$ and $f_2 = X + \pi^N$, or many 
other results.
The best we can do is to compute a GB of the fractional ideal of $\KzX$
generated by the $f_i$'s and answer that the obtained result is likely 
a GB of~$I$. In the example considered above, we will end up with the
GB $(X + O(\pi^N))$, which is certainly the more natural result we
may expect.

\subsubsection*{General log-radii.}

We now consider the case of a general $\r \in \QQ^n$.
In this situation, the monoid $\TzX[\r]$ no longer admits $\gcd$'s.
As a basic example, take $\r = (\frac 1 2, \frac 1 2)$ and consider
the terms $t_1 = \pi X_1$ and $t_2 = \pi X_2$. Then $\val_\r(t_1) = \val_\r
(t_2) = \frac 1 2$. So the valuation of $\gcd(t_1,t_2)$ 
should be $\frac 1 2$ as
well, implying that $\gcd(t_1,t_2)$ should be $\sqrt \pi$, which is
not an element of $\TzX[\r]$.
When we are working over $\KX[\r]$, this issue does not happen 
since we can freely multiply by any power of $\pi$. Over $\KX[\r]$,
Algorithm~\ref{algo:buchberger} works and is correct
(althought we have to be careful with the normalization of $\gcd$'s
in order to avoid losses of precision as much as possible).

Let us now focus on the case of $\KzX[\r]$ which is more complicated.
Let $D$ be a common denominator of the coordinates of $\r$, \emph{i.e.} 
$D{\cdot}\r \in \ZZ^n$. We consider the field extension $L = K[\eta]$ 
with $\eta^D = \pi$. The valuation $\val$ extends uniquely to $L$; we 
have $\val(\eta) = \frac 1 D$. We define $\Lz$, $\LX$ and $\LzX$ 
accordingly. Observe that $\Lz = \Kz[\eta]$.
If $D{\cdot}\r = (r_1, \ldots, r_n)$, we have $\LX[\r] = L\{\Y\}$ and 
$\LzX[\r] = L\{\Y\}^\circ$ with $Y_i = \eta^{r_i} X_i$. Moreover the 
valuation $\val_\r$ over $\LX[\r]$ (resp. $\LzX[\r]$) is transformed
into the valuation $\val_0$ over $L\{\Y\}$ (resp. $L\{\Y\}^\circ$). The
above identifications show that there is a good notion of $\gcd$'s 
and S-polynomials over $\LX[\r]$ and $\LzX[\r]$, so that eventually
Algorithm~\ref{algo:buchberger} runs and computes
GB over $\LX[\r]$ and $\LzX[\r]$. Before relating those to GB over 
$\KX[\r]$ and $\KzX[\r]$, we need to examine the shape of the GB 
output by Algorithm~\ref{algo:buchberger}.

Let $\LKX[\r]$ be the subset of $\LX[\r]$ consisting of elements of the 
form $\eta^v f$ for $v \in \NN$ and $f \in \KX[\r]$. Clearly, $\LKX[\r]$ 
is stable by multiplication. Beyond this, one can check that it exhibits 
additional stability properties:

\begin{prop}
\label{prop:LKXstab}
(1)~When Algorithm~\ref{algo:reduction} is called with inputs
$f, h_1, \ldots, h_m \in \LKX[\r]$, it outputs 
$q_1, \ldots, q_m, r \in \LKX[\r]$.

\noindent
(2)~If $f, g \in \LKX[\r]$, then $S(f,g) \in \LKX[\r]$.
\end{prop}

From Proposition~\ref{prop:LKXstab}, we deduce immediately that,
when Algorithm~\ref{algo:buchberger} is called with inputs
$f_i \in \KX[\r] \subset \LX[\r]$, the GB it outputs consists of
elements of $\LKX[\r]$. The following proposition shows that, after
minimizing this GB, we obtain a GB of the ideal of $\KzX[\r]$ we
started with.

\begin{prop}
Let $I$ be an ideal of $\KzX$.
Let $G$ be a minimal GB of $I {\cdot} \LzX[\r]$.
We assume $G \subset \LKX[\r]$.
Then $G \subset \KX[\r]$ and $G$ is a
minimal GB of $I$.
\end{prop}

\begin{proof}
Write $I_L = I {\cdot} \LzX[\r]$. We claim that:
\begin{equation}
\label{eq:LTIL}
\LT(I_L) = \eta^\N \LT(I)
\quad \text{and} \quad
I = I_L \cap \KX[\r].
\end{equation}
The inclusion $\eta^\N \LT(I) \subset \LT(I_L)$ is clear. As for
the reverse inclusion, it follows from the fact that
any $f \in I_L$ can be decomposed as
$f = f_0 + \eta f_1 + \cdots + \eta^{D-1} f_{D-1}$ with $f_i \in
\KX[\r]$ for all~$i$. Set $J = I_L \cap \KX[\r]$. From
$\LT(I_L) = \eta^\N \LT(I)$, we deduce $\LT(I) = \LT(J)$. Since
moreover $J$ obviously contains $I$, we find $I = J$.

Let $g \in G$. Write $\LT(g) = \eta^v a \X^\i$ with $v \in \NN$, $a \in
K^\times$ and $\i \in \N^n$. Since $G$ is a minimal GB of $I_L$, we
know that $\LT(g)$ is minimal in $\LT(I_L)$. From
Eq.~\eqref{eq:LTIL}, we deduce that $\LT(g) \in \TX[\r]$, that is
$\eta^v a \in K$. Thus $\eta^v \in K$ and $g \in \KX[\r]$ as claimed. 
The fact that $G$ is a minimal GB of $I$ follows again from 
Eq.~\eqref{eq:LTIL}.
\end{proof}

To conclude this section, we underline that \emph{all} computations 
(\emph{i.e.} Algorithm~\ref{algo:reduction} and the computation of 
S-polynomials) can be carried out within $\LKX[\r]$, representing an 
element of this set as a pair $(v,f)$ with $v \in \N$ and $f \in 
\KX[\r]$.
This strategy avoids constructing and working in the field $L$.

\subsection{F4 algorithm}
\label{ssec:F4}

In the history of the computation of
Gröbner bases, the development 
of Faugère's F4 algorithm \cite{F99}
has been a decisive cornerstone towards
faster algorithms.
In this section, we adjust its strategy
to the computation of Gröbner bases
over Tate algebras.
We restrict ourselves to $\r = 0$, keeping in mind that the case 
of general log-radii can be reached using the techniques discussed
at the end of \S\ref{ssec:buchberger}.

Roughly, the F4 algorithm is an adaptation
of Buchberger's algorithm such that
all S-polynomials of a given degree
are processed and reduced together in a big 
matrix of polynomials, along with their reducers.
The algorithm carries on the computation
until there is no S-polynomials to reduce.
Over Tate algebras, there is no degree
as for polynomials.
However, we can use instead
the degree of the $\lcm$ of 
the leading terms of an $S$-pair.

The F4 strategy can be then summed-up as follows:

\noindent
(1)~Collect all S-pairs sharing the smallest degree
for the $\lcm$ of their leading terms, and prepare their reduction (Algorithm~\ref{algo:SymbPreproc}).

\noindent
(2)~Reduce them all together (Algorithm~\ref{algo:matrix_reduction}).

\noindent
(3)~Update the GB in construction and list of $S$-pairs
according to the result of the previous reduction.

\noindent
(4)~Carry on the previous steps until there is
no S-pair remaining.

\noindent
The main algorithm is Algorithm \ref{algo:F4_algo_complete},
with Algorithms~\ref{algo:matrix_reduction} and~\ref{algo:SymbPreproc}
as subroutines.

\newcommand{\algoskip}{\vspace{1.9mm}}
\SetAlgoSkip{algoskip}

\begin{algorithm} 
\DontPrintSemicolon

 \SetKwInOut{Input}{input}\SetKwInOut{Output}{output}

 \Input{a matrix $M$,\\
        a list of monomials $\mon$ indexing the col. of $M$}
 \Output{the $U$-part of the Tate LUP-form of $M$}

\lIf{$M$ has no non-zero entry}{Return $M$;}
\textbf{Find} $i,j$ s.t. $M_{i,j}$ has the greatest term $M_{i,j} x^{\mon_j}$ for $\leq$;\;
\textbf{Swap} the columns $1$ and $j$ of $M$;\;
\textbf{Swap} the entries $1$ and $j$ of $\mon$;\;
\textbf{Swap} the rows $1$ and $i$ of $M$;\;
By \textbf{pivoting} with the first row, eliminates the coefficients of the other rows on the first column;\;
\textbf{Proceed recursively} on the submatrix $M_{i \geq 2, j \geq 2}$;\;
\textbf{Return} $M$;\;

 \caption{\texttt{TateRowReduction}\label{algo:matrix_reduction}}
\end{algorithm}

\begin{algorithm} 
\SetKwInOut{Input}{input}\SetKwInOut{Output}{output}
\SetKwComment{tcp}{\rm \color{comment} \#\, }{}
\Input{a list $P$ of pairs of elements of $\KX$ (resp. of $\KzX$),\\
 a list $G$ of elements in $\KX$ (resp. in $\KzX$).}
\Output{a matrix $M$}
$U \leftarrow $ the series in $P$\;
$C \leftarrow \bigcup_{f \in U} \{\text{terms of } f\}$\;
$\mathfrak A \leftarrow K$ (resp. $\mathfrak A \leftarrow \Kz$);\,
$D \leftarrow \emptyset $\;
\While{$\mathfrak A {\cdot} C \neq \mathfrak A {\cdot} D$}{
$t \leftarrow \max \, \{ t\in C, \, t \not\in \mathfrak A {\cdot} D\}$\;
$D \leftarrow D \cup \{ t \}$\;
$V \leftarrow \big\{\big(g, \frac{t}{LT(g)}\big) \text{ for } g \in G \text{ s.t. }
LT(g) \mid t \big\}$\;
\If{$V \neq \emptyset$}
{
$(g, \delta) \leftarrow$ the element $(g, \delta)$ of $V$ with maximal
$\LT(\delta{\cdot}g)$, with tie-breaking by taking minimal $\delta$ 
(for degree then for $\leq_\omega$)\; 
$U \leftarrow U \cup \{ \delta{\cdot}g \}$\;
$C \leftarrow C \cup \{ \text{terms of } \delta{\cdot}g\}$\;
}
}
$M \leftarrow$ the series of $U,$ written in a matrix of series\;			
\textbf{Return} $M$\; 
 \caption{\texttt{Symbolic-Preprocessing}} \label{algo:SymbPreproc}
\end{algorithm}

\begin{algorithm} 
\DontPrintSemicolon

 \SetKwInOut{Input}{input}\SetKwInOut{Output}{output}

  \Input{$f_1,\dots,f_m$ in $\KX$ (resp. in $\KzX$)}
  \Output{a GB $G$ of the ideal of $\KX$ (resp. of $\KzX$) 
  generated by the $f_i$'s}

$G \leftarrow (f_1,\dots, f_m)$;\;
$B \leftarrow \{ (f_i,f_j), \, 1 \leq i < j \leq m \}$;\;
\While{$B \neq \emptyset$}{
$d \leftarrow \min_{(u,v) \in B} \deg \lcm(\LT(u),\LT(v))$;\;
$P$ \textbf{receives} the pop of the pairs of degree $d$ in $B$;\;
$M \leftarrow \texttt{Symbolic-Preprocessing}(P,G)$;\;
$M \leftarrow \texttt{TateRowReduction}(M)$;\;
\textbf{Add} to $G$ all the polynomials obtained from $M$ that provide leading terms not in $\left\langle \left\lbrace  LT(g) \text{ for } g \in G \right\rbrace \right\rangle$;\;
\textbf{Add} to $B$ the corresponding new pairs;\;
}			
\textbf{Return} $G$;\; 

 \caption{F4 algorithm} \label{algo:F4_algo_complete}
\end{algorithm}

\begin{lem}
At finite precision,
Algorithm \ref{algo:SymbPreproc} terminates
in a finite number of steps, 
and the output $M$
has a finite number of rows. \label{lem:Symb_Preproc_terminates}
\end{lem}

\begin{proof}
We remark that the sequence formed by the elements $t$'s 
considered n the while loop is strictly decreasing.
Indeed, we notice first that $t$ is added to $D$ on line 6,
so it cannot reappear later.
Then, if $V$ is not empty, 
all the terms of $\delta{\cdot}g$
on line 11 are strictly smaller than $t$,
except its leading term which is $t$.
At finite precision, 
there is no infinite strictly decreasing sequence by
Lemma~\ref{lem:strictly_dec_seq}.
Consequently, Algorithm~\ref{algo:SymbPreproc} terminates
in a finite number of steps.
\end{proof}

\begin{prop}
Under the same hypotheses as in Theorem \ref{theo:prec},
Algorithm \ref{algo:F4_algo_complete} outputs
$G$ satisfying the same conclusions.
\end{prop}

\begin{proof}
Thanks to Lemma \ref{lem:Symb_Preproc_terminates},
it is clear that Algorithms~\ref{algo:matrix_reduction} 
and~\ref{algo:SymbPreproc} terminate.
Termination of Algorithm \ref{algo:F4_algo_complete}
can then be proved along the following lines.
If the algorithm did not terminate for some given input,
then it would mean that $B$ (the list of pairs) is never empty.
Hence, there would be an infinite number of times
when new polynomials are added to $G$. From them, we would be 
able to construct a strictly increasing sequence of
monomial ideals inside $\TX$ which are nonzero at the
precision $O(\pi^N)$. This contradicts Lemma~\ref{lem:strictly_dec_seq}.
Finally, thanks to the Buchberger criterion for
Tate algebras (\emph{cf}~Theorem~\ref{theo:Buchberger_criterion}),
the correctness follows along the same lines as in the proof 
of Theorem~\ref{theo:prec}.
\end{proof}

\section{Implementation}
\label{sec:impl}

We have implemented in \sage all the algorithms presented in this 
paper, together with an interface for working with Tate algebras.
Our implementation of Buchberger algorithm (\emph{cf} 
\S\ref{ssec:buchberger}) is now part of the standard distribution of 
\sage since version 8.5.
It is fairly optimized but it is clear that more work need to be
done in this direction: the timings we obtain are far from the
average timings reached by other softwares (as \textsc{singular}) 
for the computation of Gröbner bases over $\ZZ/p^n\ZZ$, whereas 
we could expect them to match, even if the context is a bit 
different.
Our implementation of the F4 algorithm (\emph{cf} 
\S\ref{ssec:F4})
is still a toy implementation, 
which does not exhibit good performances yet; we plan to improve it
in a near future. It is available at:

\smallskip

\noindent
\hfill
\url{https://gist.github.com/TristanVaccon}
\hfill\null

\subsubsection*{Short demo.}

Our implementation provides a constructor for creating Tate algebras, 
called {\color{constructor} \verb?TateAlgebra?}:

\smallskip

{\noindent \small
\begin{tabular}{rl}
\cIn
 & {\color{parent}\verb?K?}\verb? = ?{\color{constructor}\verb?Qp?}\verb?(2, ?{\color{keyword}\verb?prec?}\verb?=5, ?{\color{keyword}\verb?print_mode?}\verb?='?{\color{string}\verb?digits?}\verb?')? \\
 & {\color{parent}\verb?A?}\verb?.<x,y> = ?{\color{constructor}\verb?TateAlgebra?}\verb?(?{\color{parent}\verb?K?}\verb?); ?{\color{parent}\verb?A?} \\
\cOut
 & $\QQ_{2}\{x,y\}$ 
\end{tabular}}

\smallskip

\noindent
We observe that, by default, the log-radii are all zero; the keyword 
{\color{keyword} \verb?log_radii?} can be use to pass in other values. 
Similarly the default order is the one attached to $\omega = 
\text{grevlex}$, but any other order known by \sage can be specified 
\emph{via} the keyword {\color{keyword} \verb?order?}.

The ring of integers of the Tate algebras can be built as follows:

\smallskip

{\noindent \small
\begin{tabular}{rl}
\cIn
 & {\color{parent}\verb?Ao?}\verb? = ?{\color{parent}\verb?A?}\verb?.?{\color{method}\verb?integer_ring?}\verb?(); ?{\color{parent}\verb?Ao?} \\
\cOut
 & $\QQ_{2}\{x,y\}^\circ$
\end{tabular}}

\smallskip

\noindent
We can now create and manipulate elements:

\smallskip

{\noindent \small
\begin{tabular}{rl}
\cIn
 & \verb?f = 2*x^2 + 5*x*y^2? \\
 & \verb?g = 4 + 2*x^2*y? \\
 & \verb?f + g? \\
\cOut
 & $...00101 x y^2 + ...000010 x^2 y + ...000010 x^2 + ...0000100$ \\
\cIn
 & \verb?(1+g).?{\color{method}\verb?inverse_of_unit?}\verb?()? \\
\cOut
 & $...01101 + ...01110x^{2}y + ...10100x^{4}y^{2} + {}$ \\
 & $...11000x^{6}y^{3} + ...10000x^{8}y^{4} + O(2^{5} \: \QQ_{2}\{x,y\}^{\circ})$
\end{tabular}}

\smallskip

\noindent
We observe that, in the outputs, terms are ordered with respect to
the term order on $\TX$, the greatest one coming first.
The big-oh appearing on the last line hides terms which are multiple 
of $2^5$.

\vspace{1cm}

Classical transcendantal functions are also implemented, \emph{e.g.}:

\smallskip

{\noindent \small
\begin{tabular}{rl}
\cIn
 & {\color{method}\verb?log?}\verb?(1+g)? \\
\cOut
 & $...01110x^4y^2 + ...11010x^2y + ...11100x^8y^4 + $ \\
 & $...11100 + ...11000x^6y^3 + O(2^{5} \: \QQ_{2}\{x,y\}^{\circ})$
\end{tabular}}

\smallskip

\noindent
Ideals of $\KX$ can be defined and manipulated as follows:

\smallskip

{\noindent \small
\begin{tabular}{rl}
\cIn
 & {\color{parent}\verb?J?}\verb? = ?{\color{parent}\verb?A?}\verb?.?{\color{method}\verb?ideal?}\verb?([f,g])? \\
 & {\color{parent}\verb?J?}\verb?.?{\color{method}\verb?groebner_basis?}\verb?()? \\
\cOut
 & \verb?[ ? $...0001x^3 + ...1011y + O(2^{4} \: \QQ_{2}\{x,y\}^{\circ})$, \\
 & \verb?  ? $...00001x^2y + ...00010 + O(2^{5} \: \QQ_{2}\{x,y\}^{\circ})$, \\
 & \verb?  ? $...0001y^2 + ...1010x + O(2^{4} \: \QQ_{2}\{x,y\}^{\circ})$ \verb? ]? \\
\cIn
 & {\color{parent}\verb?A?}\verb?.?{\color{method}\verb?random_element?}\verb?()*f + ?{\color{parent}\verb?A?}\verb?.?{\color{method}\verb?random_element?}\verb?()*g ?{\color{keyword}\verb?in?}\verb? ?{\color{parent}\verb?J?} \\
\cOut
 & True \\
\cIn
 & {\color{method}\verb?log?}\verb?(1+g) ?{\color{keyword}\verb?in?}\verb? ?{\color{parent}\verb?J?} \\
\cOut
 & True \\
\end{tabular}}

\smallskip

\noindent
And similarly for ideals of $\KzX$ (observe that no losses of
precision occur this time, in accordance with Theorem~\ref{theo:prec}):

\smallskip

{\noindent \small
\begin{tabular}{rl}
\cIn
 & {\color{parent}\verb?Jo?}\verb? = ?{\color{parent}\verb?Ao?}\verb?.?{\color{method}\verb?ideal?}\verb?([f,g])? \\
 & {\color{parent}\verb?Jo?}\verb?.?{\color{method}\verb?groebner_basis?}\verb?()? \\
\cOut
 & \verb?[ ? $...00001xy^2 + ...11010x^2 + O(2^{5} \: \QQ_{2}\{x,y\}^{\circ})$, \\
 & \verb?  ? $ ...000010x^2y + ...000100 + O(2^{6} \: \QQ_{2}\{x,y\}^{\circ})$, \\
 & \verb?  ? $ ...000100x^3 + ...101100y + O(2^{6} \: \QQ_{2}\{x,y\}^{\circ})$, \\
 & \verb?  ? $ ...000100y^2 + ...101000x + O(2^{6} \: \QQ_{2}\{x,y\}^{\circ})$ \verb? ]? \\
\cIn
 & \verb?g/2 ?{\color{keyword}\verb?in?}\verb? ?{\color{parent}\verb?Jo?} \\
\cOut
 & False \\
\end{tabular}}

\begin{small}

\bibliographystyle{plain}

\end{small}

\end{document}
